\documentclass{amsart}
\usepackage{euscript}           
\usepackage{amsmath,amsthm}     
\usepackage{amssymb}            
\usepackage[usenames,dvipsnames]{color}
\usepackage{graphicx}

\newlength{\labwidth}
\newcommand{\labarrow}[1]{
\settowidth{\labwidth}{$\scriptstyle \;\; #1 \;\;$}
\stackrel{#1}{\smash{\hbox to \labwidth{\rightarrowfill}}
\vphantom{\longrightarrow}}
}

\newcommand{\bP}{{\mathbf P}}

\renewcommand{\S}{{\mathbb S}}

\newcommand{\m}{{\mathfrak m}}

\usepackage {epsf}
\DeclareMathOperator*{\colim}{colim}

\newcommand{\comment}[1]{}

\newcommand {\Z}{{\mathbb Z}}

\newcommand {\TM}{{\mathbb{TMF}_1(3)}}
\newcommand {\T}{{\mathbb T}}
\newcommand {\U}{{\mathbb U}}
\newcommand {\KK}{{\mathbb{M}}}
\newcommand {\E}{{\mathbb E}}

\newcommand {\B}{{\mathbb B}}
\newcommand {\bbP}{{\mathbb P}}
\newcommand {\bbM}{{\mathbb M}}
\newcommand {\CC}{{\mathbb C}}

\newcommand{\lra}{\longrightarrow}              

\newfont{\german}       {eufm10 at 12pt}

\DeclareMathOperator{\mfgl}{\star}

\numberwithin{equation}{section}
\newtheorem{thm}[equation]{Theorem}
\newcounter{numerierer}
\newcounter{leer}

\newtheorem{defn}[equation]{Definition}
\newtheorem{prop}[equation]{Proposition}

\newtheorem{cor}[equation]{Corollary}
\newtheorem{lemma}[equation]{Lemma}

\theoremstyle{definition}  

\newtheorem{set theory}[equation]{Set Theoretic Prelude}

\newtheorem{convention}[equation]{Convention}

\newtheorem{remark}[equation]{Remark}
\usepackage[frame,matrix,curve, arrow]{xy}

\xymatrixcolsep{1.9pc}                          
\xymatrixrowsep{1.9pc}
\newdir{ >}{{}*!/-5pt/\dir{>}}                  
\newdir{ |}{{}*!/-5pt/\dir{|}}                  

\subjclass{}
\begin{document}

\title{$TMF_0(3)$--Characteristic Classes for String Bundles}
\author{Gerd Laures and Martin Olbermann}

\address{ Fakult\"at f\"ur Mathematik,  Ruhr-Universit\"at Bochum, NA1/66, D-44780 Bochum, Germany}


\subjclass[2000]{Primary 55N34; Secondary 55P20, 22E66}

\date{\today}

\begin{abstract}
We compute the completed $TMF_0(3)$--cohomology of the 7--connected cover $BString$ of $BO$. We use cubical structures on line bundles over elliptic curves to construct an explicit class which together with the Pontryagin classes freely generates the cohomology ring.
\end{abstract}

\maketitle

\section{Introduction and statement of results}
Characteristic numbers play an important role in the determination of the structure of cobordism rings. For unoriented, oriented and spin manifolds, the cobordism rings were calculated in the 50s and 60s by Thom \cite{MR0061823}, Novikov \cite{MR0157381} and Anderson, Brown and Peterson \cite{MR0219077} with the help of Stiefel--Whitney, $H\Z$-- and $KO$--Pontryagin numbers. However, it is known that for manifolds with lifts of the tangential structure to the 7--connected cover $BString$ of $BO$ these numbers do not determine the bordism classes.\par
Locally at the prime 2, the Thom spectrum $MSpin$ splits into summands of connected covers of $KO$ and an Eilenberg-MacLane part \cite[Theorem 2.2]{MR0219077}. A similar splitting is conjectured for $MString$ where $KO$ is replaced by suitable versions of the spectrum $TMF$: the Witten orientation provides a surjection of the string bordism ring to the ring of topological modular forms \cite[Theorem 6.25]{MR1989190} and there is evidence \cite[Section 7]{MR2508904} that another summand of $MString $ is provided by the 16--connected cover of $TMF_0(3)$. In order to provide maps to this possible summand one has to study  $TMF_0(3)$--characteristic classes for string manifolds which is the subject of this work.\par
In \cite{L13} the $TMF_1(3)$--cohomology rings of $BSpin$ and $BString$ were computed. It turned out \cite[Theorem 1.2]{L13} that the $TMF_1(3)$--cohomology ring of $BSpin$ is freely generated by certain classes $p_i$. These deserve the name Pontryagin classes since they share the same properties as the $H\Z$-- and $KO$--Pontryagin classes. In the case of $BString$, there is another class $r$ coming up which together with the Pontryagin classes freely generates the $TMF_1(3)$--cohomology ring when localized at $K(2)$ for the prime 2. \par
The theory $TMF_1(3)$ is a complex orientable theory. Its formal group is the completion  of the universal elliptic curve with $\Gamma_1(3)$--structures. Its relation to $TMF_0(3)$ is analogous to the relation between complex and real $K$--theory:  a $\Gamma_1(3)$--structure is a choice of point of exact order 3 on an elliptic curve. A $\Gamma_0(3)$--structure is the choice of  subgroup scheme of the form $\Z/3$ of the points of order 3. Given such a subgroup scheme there are exactly two choices of  points of exact order 3 and they differ by a sign. Hence the corresponding cohomology theory $TMF_0(3)$ is the ``real'' version of the complex theory $TMF_1(3)$. It can be obtained by taking homotopy fixed points under the action which changes the sign of the 3 division point. It is not a complex orientable theory.
\par
The theory $TMF_1(3)$ admits a genuine equivariant $\Z/2$--action which turns it into  a Real theory in the sense of Atiyah \cite{MR0206940}, meaning an equivariant theory where the $\Z/2$--action is induced by complex conjugation. We construct this action in Section \ref{RealTMFPont} with the help of the work of Hu and Kriz \cite{MR1808224}. The associated fixed point spectrum is $TMF_0(3)$.
This allows us to lift the Pontryagin classes to $TMF_0(3)$ for $Spin$--bundles. Our first result is:
\begin{thm}\label{main Spin}
There are classes  $\pi _i\in TMF_0(3)^{-32i}BSpin$ which lift the products $v_2^{6i}p_i$ for the $TMF_1(3)$--Pontryagin classes $p_i$. Moreover, we have
$$ TMF_0(3)^{*}BSpin\cong TMF_0(3)^*[\![ \pi _1,\pi _2,\ldots]\!].$$
\end{thm}

The generator $r$ in the calculation of the  $TMF_1(3)$--cohomology of $BString$ has the property that it maps to a generator of $K(\Z,3)$ under the canonical map. In fact, it has been shown in \cite[Theorem 1.3]{L13} that any class with this property can serve as a generator. However, in order to obtain a class $r$ which is already defined in the theory $TMF_0(3)$ one has to provide a more geometric construction. We use the theory of cubical structures on elliptic curves which also played a role in \cite{ MR1869850} in the construction of the  Witten orientation. We show that a convenient choice of a generator $r$ is the difference class which compares the Witten orientation with the complex orientation. It turns out that this class admits a lift to the equivariant setting.
 Our main result is:
\begin{thm}\label{main}
For a string bundle $\xi$ over a space $X$ there is a natural stable class $$r(\xi) \in TMF_0(3)^0 X$$ with the following properties:
\begin{enumerate}
\item
$r$ is multiplicative: $r(\xi\oplus \eta)=r(\xi)\otimes r(\eta)$.
\item
There is an isomorphism
$$ \widehat{TMF_0(3)}^*[\![\tilde{r}, \pi _1,\pi _2,\ldots]\!] \lra \widehat{TMF_0(3)}^*BString$$
where $\widehat{TMF_0(3)}$ denotes $(L_{K(2)}TMF_1(3))^{h\Z/2}$
and, in abuse of notation, the class $\tilde{r}$ is the $K(2)$--local and reduced version of the class $r$ corresponding to the universal bundle over $BString$.
\item
The Chern character of the elliptic character \cite{MR1022688} of $r$
is given by the formula
$$ ch(\lambda(r(\xi))) = \prod_i\frac{\Phi(\tau,x_i-\omega)}{\Phi(\tau,-\omega)}
$$
where the $x_i$ are the formal Chern roots of $\xi\otimes \CC$, $\omega =2\pi i/3$ and $\Phi$ is the theta function
\begin{eqnarray*}
\Phi(\tau,x)&=&(e^{x/2}-e^{-x/2})\prod_{n=1}^\infty \frac{(1-q^ne^x)(1-q^ne^{-x})}{(1-q^n)^2}
\\&=&x\, \exp(-\sum_{k=1}^\infty \frac{2}{(2k)!}G_{2k}(\tau)x^{2k}).
\end{eqnarray*}

\end{enumerate}
\end{thm}
The paper is organized as follows: in Section \ref{RealTMFPont} we first remind the reader of the theory $TMF_1(3)$ and construct its Real version in the category of $\Z/2$--equivariant spectra. Its homotopy fixed points spectrum provides a model for $TMF_0(3)$. This allows us to construct Pontryagin classes in the equivariant setting for string manifolds. Theorem \ref{main Spin} then follows from the homotopy fixed point spectral sequence. In the third section we use cubical structures to show that the difference class of two $MU\langle 6 \rangle $--orientations of $TMF_1(3)$  can be taken as the remaining generator in the cohomology of $BString$. This involves the computation of a specific coefficient of a  cubical structure which is carried out in Appendix B. In the fourth section, this class 
is  lifted to the equivariant setting. Since the realification of the difference class $r$ provides a permanent cycle in the homotopy fixed point spectral sequence the main theorem follows.
  
\subsubsection*{Acknowledgements.} The authors would like to thank Nitu Kitchloo, Bj\"orn Schuster and  Vesna Stojanoska  for helpful discussions. They are also grateful to the referee for a very careful revision.

\section{Real topological modular forms and Pontryagin classes} \label{RealTMFPont}
In this section we will construct Pontryagin classes for the $\Z/2$--equivariant spectrum of $\Gamma_1(3)$--topological modular forms  and prove the first theorem. We start reviewing level structures for elliptic curves and proceed with the associated spectra of topological modular forms. The method of Hu and Kriz (compare \cite{MR1808224}) allows a construction of $TMF_1(3)$ in the category of $\Z/2$--equivariant spectra.\par
Let ${\mathcal M}$ be the stack of smooth elliptic curves. A morphism $f:S\lra {\mathcal M}$ determines an elliptic curve $C_f$ over the base scheme $S$, see Deligne and Rapoport  \cite{MR0337993}.  There is the line bundle $\omega$ of invariant differentials on ${\mathcal M}$.  A modular form of weight $k$ is section of $\omega^{\otimes k}$. \par
For an elliptic curve $C$ over a ring $R$  let $C[n]$ denote the kernel of the self map  $[n]$ which multiplies by $n$ on $C$. If $n$ is invertible in $R$ then \'{e}tale-locally $C[n]$ is of the form $\Z/n\times \Z/n$. A choice of a subgroup scheme $A$ of $C[n]$ which is isomorphic to $\Z/n$ is called a $\Gamma_0(n)$--structure and a choice of a monomorphism from $\Z/n$ to $ C[n]$ is a $\Gamma_1(n)$--structure. 
\par
Moduli problems for $\Gamma_1(n)$--structures with $n=3$ or $n\geq 5$ are representable by a smooth affine curve whereas for $\Gamma_0(n)$--structures they are not
(compare \cite[Corollary 4.7.1]{MR0337993}).  The case $n=3$ can be made more explicit: locally, a curve $C$ can uniquely be written in the form 
 \begin{eqnarray}\label{curve}
 & C: &y^2 +a_1xy +a_3 y = x^3
 \end{eqnarray}
 in a way that the distinguished point $P$ of order 3 is the origin $(0,0)$ and the invariant differential has the standard form $\omega=dx/(2y+a_1x+a_3)$. This means that 
 the ring of modular forms of level 3 has the form 
 \begin{eqnarray}\label{mf}
  {M_{\Gamma_1(3)}}_*&\cong &\Z[1/3, a_1,a_3,\Delta^{-1}]
  \end{eqnarray}
where $\Delta=a_3^3(a_1^3-27a_3)$. See \cite[Proposition 3.2]{MR2508904} for more details. \par
 For each \'{e}tale map $f: Spec( R ) \lra {\mathcal M}$ with elliptic curve $C_f$ there is a spectrum $E$ which is a complex orientable ring spectrum whose formal group ${\rm Spf} E^0BU(1)$ is equipped with an isomorphism to the formal completion of $C_f$, see \cite{MR2648680}. The spectrum $TMF_1(3)$ can be obtained this way. Its  coefficient ring is $M_{\Gamma_1(3)}{}_*$ and it carries the universal curve $C$ together with the universal point $P$ of order 3.\par
Let $MU$ be the complex bordism spectrum in the 2--local category. Choose $x_i\in\pi_{2i}MU$ which generate the bordism ring 
$$ MU_*\cong \Z_{(2)}[x_1,x_2,\ldots].$$
Let
$$ \varphi: MU_*\lra  M_{\Gamma_1(3)}{}_*$$
be the map which classifies the formal group law associated to the Weierstrass curve \eqref{curve}  with its standard coordinate as described in \cite[Chapter IV]{MR2514094}. Note that this map already takes values in the ring $\Z_{(2)}[a_1,a_3]$. Define polynomials $f_i\in \Z_{(2)}[s,t]$ in two variables $s,t$ by  $$\varphi(x_i)=f_i(a_1,a_3).$$  We may assume $f_1=s$, $f_3=t$ since for $p=2$ the Hazewinkel generators are $v_1=a_1$ and $v_2=a_3$ (c.\cite[Lemma 1]{MR2076927}). 
The sequence $r_i=x_i-f_i(x_1,x_3)$ for $i=2,4,5,6\ldots$ is regular in $MU_*$ and the map
$$ MU_*/r_1,r_2,\ldots \lra TMF_1(3)_*$$
is an isomorphism after inverting $\Delta$. Hence, $TMF_1(3)$ is a (strict) $MU$--module which is obtained by killing off a regular sequence and localizing at an element of $MU_*$ (see \cite[Chapter V]{MR1417719}).

If we invert $v_2$ instead of $\Delta$,  we obtain another module with homotopy $\Z_{(2)}[v_1,v_2^{\pm 1}]$. Both modules are generalized Johnson--Wilson spectra $E(2)$ 
in the sense of \cite[Definition 2.3]{L13} since $\Delta$ coincides with a power of $v_2$ modulo the ideal $(2,a_1)$ , compare \cite[Proposition 2.4]{L13}. 
Many theorems about the standard $E(2)$ carry over to generalized $E(2)$ --- often the vanishing of the higher $v_k$ is not needed in the proofs.
\begin{convention}\label{convention}
In the sequel it is not important if  the discriminant $\Delta$ is inverted instead of $v_2$ only. All of our statements hold for both modules and we will not distinguish between them in the notation. The inverted class will be called the periodicity class $P$ in the sequel, and is a polynomial in $v_1$ and $v_2$. Also, all spectra and homotopy groups will be localized 
at the prime $2$. 
\end{convention}
\par
Next we would like to construct the $\Z /2$--equivariant version of $TMF_1(3)$. For an object $\E$ in the category of $\Z/2$--equivariant spectra over a complete universe, we write $\E_*$ for the coefficients in integral dimensions and we write $\E_\star$ for the coefficients in all dimensions $k+\l\alpha$ for $k,l\in \Z$ and $\alpha$ the sign representation. Let $\bbM \U$ be the Real bordism spectrum. Hu and Kriz \cite[Proposition 2.27]{MR1808224} showed that the canonical map from $\bbM  \U_\star$ to $MU_*$ splits by a map of rings $\iota$ which sends the generator $x_i$ of $MU_*$ to a class of dimension $i(1+\alpha)$. 

Moreover, the action of $\Z/2$ on $MU$ which changes the coordinate $x$ to $[-1](x)$ corresponds to the $\Z/2$--action on $TMF_1(3)$ with homotopy fixed point set $TMF_0(3)$: it comes from the map on the moduli stack which takes the point of order 3 to its negative. A detailed argument is given in Appendix \ref{GeomInvariance}.

As in \cite[p.332]{MR1808224}  we may kill the sequence $\iota(r_1), \iota(r_1),\ldots$ in $\bbM \U _\star$ (with $r_i$ defined above) in the category of  $\bbM \U$--module spectra and invert the periodicity class. We have constructed the Real spectrum $\TM$.

It will be useful to construct an equivariant version of the $K(2)$--localization $\widehat{TMF_1(3)}=L_{K(2)}TMF_1(3)$,
whose coefficients are $\widehat{TMF_1(3)}^*\cong (TMF_1(3)^*)^{\wedge}_{(2,v_1)}$.
Set $$\KK(2,\iota(v_1))=\KK(2)\wedge _{\bbM \U}\KK(\iota(v_1)),$$
where the Koszul spectrum $\KK(\beta)$ is the fibre of the canonical map from $\bbM \U$ to $\bbM \U[\beta^{-1}]$. Define the $\bbM \U$--function spectrum 
\begin{align*}
\widehat{\TM}& = F_{\bbM \U}(\KK(2,\iota(v_1)), \TM)\\
& \cong  F_{\bbM \U[\iota(v_2)^{-1}]}(\KK(2,\iota(v_1))[\iota(v_2)^{-1}], \TM).
\end{align*}  
It follows from \cite[Chapter XXIV, Theorem 5.3, (5.2)]{MR1413302} that $\widehat{\TM}$ is the Bousfield localization with respect to $\bbM\U/(2,\iota(v_1))[\iota(v_2)^{-1}]$ and its coefficient ring is given by the completion with respect  
to the ideal $I=(2,\iota(v_1))$.

\begin{lemma} \label{SCT}
Let $i^*$ be the forgetful functor from $\Z/2$--equivariant to non-equivariant spectra.
\begin{enumerate}
\item
$i^*\TM \cong TMF_1(3)$
\item
 $i^*\widehat{\TM} \cong\widehat{TMF_1(3)}$  
\item 
$\TM^{\Z /2}\cong \TM^{h\Z /2}\cong  TMF_0(3)$
\item
 $\widehat{\TM}^{\Z /2}\cong \widehat{\TM}^{h\Z /2}\cong  \widehat{TMF_0(3)}$
\end{enumerate}
\end{lemma}
\begin{proof}
Non-equivariantly, we have $i^*\bbM \U \cong MU$, so the first statement follows from what we said before. For the second statement, 
we also need the fact \cite[Theorem 2.1(a)]{MR737778} that
$K(2)$--localization can be replaced by localization with respect to the spectrum $MU/(2,v_1)[v_2^{-1}]$.

For the third one, observe that 
the spectrum $\TM$ satisfies the strong completion theorem, that is, we can replace it with $c{\TM }=F(E\Z/2_+,\TM )$. 
This has been shown for classical Johnson--Wilson spectra by Kitchloo and Wilson in \cite[Appendix A]{MR3044602} using results from \cite{MR1808224}. It is easy to check that the arguments given there also apply to all generalized
$E(2)$ for which $v_2$ is invertible.
It follows that the homotopy fixed point spectrum $\TM^{h\Z /2}$ coincides with the fixed point spectrum $\TM^{\Z /2}$ up to homotopy. 
For the last statement observe that we have equivalences
\begin{eqnarray*}
 \widehat{\TM}& \cong & F_{\bbM \U}(\KK(2,v_1), F(E\Z/2_+,\TM))\\
 &\cong& F_{\bbM \U}(\KK(2,v_1), F_{\bbM \U}(\bbM \U \wedge E\Z/2_+,\TM))\\
 &\cong& F_{\bbM \U}(\bbM \U \wedge E\Z/2_+, F_{\bbM \U}(\KK(2,v_1), \TM))\\
&\cong&F(E\Z/2_+,\widehat{\TM}).
\end{eqnarray*}
\end{proof}

Since $\TM\cong F(E\Z/2_+,\TM)$ there is a Borel cohomology spectral sequence \cite[Section 2]{MR1831351}
 $$E_1=C^r(\Z/2,F(\Z/2_+,\TM)^{s+t\alpha})\Rightarrow \TM^{r+s+t\alpha}.$$
The integral degree part of the spectral sequence can be identified with the homotopy fixed point spectral sequence 
 $$E_1=C^r(\Z/2,TMF_1(3)^s)\Rightarrow (TMF_1(3)^{h\Z/2})^{r+s}$$
which Mahowald and Rezk \cite{MR2508904} used  to compute $\pi_*TMF_0(3)$.
However, when we use the whole $(\Z+\Z\alpha)$--grading, the multiplicative structure of the spectral sequence is easier to understand.
  
One has $F(\Z/2_+,\TM)^\star \cong TMF_1(3)^*[\sigma^{\pm 1}]$, where
$\sigma$ is the composite
$$ \Sigma^{-1 +\alpha}\Z /2_+=\Sigma^{-1}(S^\alpha \wedge \Z /2_+)\cong \Sigma ^{-1}(S^1 \wedge \Z/2_+)=\Z/2_+\lra pt_+=S^0\lra \TM$$
of degree $1-\alpha$.
Thus we have an isomorphism 
\begin{align*}
E_1\cong  TMF_1(3)^*[\sigma^{\pm 1},a],
\end{align*}
where $a$ has degree $\alpha$. By the calculation given in Appendix \ref{GeomInvariance},  the group $\Z/2$ acts by $a_i\mapsto -a_i$.
In the spectral sequence we have permanent cycles $a,\sigma^{\pm 8}$, $v_1\sigma$ which corresponds to $\iota(v_1)\in\TM^{-1-\alpha}$ and
$v_2\sigma^3$ which corresponds to $\iota(v_2)\in\TM^{-3-3\alpha}$. We will use the notation $\iota(v_1)$ respectively 
$\iota(v_2)$ also for the cycles $v_1\sigma$ and $v_2\sigma^3$ respectively.
For later use we note that the permanent cycle $v_2^3\sigma$ corresponds to the invertible class $y=\iota(v_2)^3 \sigma^{-8}\in\TM^{-17-\alpha}$.
The non-zero differentials are induced by $d_1(\sigma^{-1})=2a, d_3(\sigma^{-2})=\iota(v_1)a^3, d_7(\sigma^{-4})=\iota(v_2)a^7$.
All of this follows from \cite[Section 2]{MR1831351} using naturality of the Borel spectral sequence.

\bigskip

Let  $p_i\in TMF_1(3)^{4i}B{\mathit Spin}$ for $i=1,2\ldots$ be the Pontryagin classes. Recall from \cite[Theorem 1.2]{L13} that they are uniquely determined by the following property: the series
$p(t)=1+p_1t+p_2t^2\ldots$ is given by $$ \prod_{i=1}^m (1-t\rho^*(x_i \overline{x}_i))$$
when restricted to the classifying space of each maximal torus of $Spin(2m)$. Here, $\rho$ is the map to the maximal torus of $SO(2m)$ and the $x_i$ (and  $\overline{x}_i$) are the first $TMF_1(3)$--Chern classes of the canonical line bundles $L_i$ (resp.\ $\overline{L}_i$) over the classifying spaces of the tori. 

We can use the  $\bbM \U $--orientation of $ \TM$  to define Pontryagin classes in $\TM^{\Z/2}$--cohomology. However, the construction is less explicit than in the complex case. We start with some preparations. 
\begin{lemma} \label{surjbubspin}
The composite
$$g: \xymatrix{BSpin \ar[r] & BSO \ar[r]& BU}$$
is surjective in $\widehat{TMF_1(3)}$--cohomology.
\end{lemma}  
\begin{proof}
Recall from \cite[Theorem 1.1(iii)]{MR1909866} and \cite[(1.6)]{MR2093483}
that both $BU$ and $BSpin$ have even $K(2)$--homology and that the map $g$ is injective in $K(2)$--homology, thus surjective in $K(2)$--cohomology.
In particular, this holds for the generalized $K(2)$ given by $\widehat{TMF_1(3)}/(2,v_1)$, see also \cite[Remark 3.1]{L13}. 

By \cite[Proposition 2.5]{MR1601906}  (which also applies to generalized $E(2)$) it follows that $\widehat{TMF_1(3)}^*(BU)$ and $\widehat{TMF_1(3)}^*(BSpin)$
are pro-free, i.e. they are completions of free $R=\widehat{TMF_1(3)}^*$-modules with respect to the ideal $\m=(2,v_1)$. We also deduce that $K(2)^*X\cong \widehat{TMF_1(3)}^*(X) /(2,v_1)$ for $X=BU$ and $X=BSpin$.

Now it suffices to prove that any homomorphism $f:M\to N$ of pro-free $R$-modules which is surjective modulo $\m$ is surjective. 
This is Nakayama's lemma. It follows by tensoring $f$ with the short exact sequences $\m^k/\m^{k+1} \to R/\m^{k+1}\to R/\m^k$ and applying induction,
noting that $\m^k/\m^{k+1}$ is a free $R/\m$-module.
\end{proof}
\begin{prop}\label{equivariant Pontryagin}
The  images of the classes 
$$\pi_i=v_2^{6i}p_i\in TMF_1(3)^{-32i}BSpin$$ in $\widehat{TMF_1(3)}^{-32i}BSpin$ lift to $\widehat{TMF_0(3)}^{-32i}BSpin$.
\end{prop}
\begin{proof}
We abbreviate $E={TMF_1(3)}$. 
The proof is given by the commutative diagram
\begin{eqnarray*}
\xymatrix{
\left(\iota(P)^{-1}\widehat{\B \bbP} \right)^{2i(1+\alpha)}\B\U
\ar[r]\ar[d]^{\cong}
&
\widehat\E^{2i(1+\alpha)}\B\U
\ar[r]^{\cdot y^{2i}}_{\cong}\ar[d]
&
\widehat\E^{-32i}\B\U
\ar[d]\ar[dr]
\\
\left(P^{-1}\widehat{BP}\right)^{4i}BU
\ar@{->>}[r]
&
\widehat E^{4i}BU
\ar[r]^{\cdot v_2^{6i}}_{\cong}\ar@{->>}[d]
&
\widehat E^{-32i}BU
\ar[d]
&
\left(\widehat\E^{\Z/2}\right)^{-32i}BSpin
\ar[dl]
\\
&
\widehat E^{4i}BSpin
\ar[r]^{\cdot v_2^{6i}}_{\cong}
&
\widehat E^{-32i}BSpin,
}
\end{eqnarray*}
which we will explain below. We lift the class $\pi_i\in \widehat E^{-32i}BSpin$ along the bottom left edge to the top left corner, and then the image along the right top edge in $\left(\widehat \E^{\Z/2}\right)^{-32i}BSpin$ is the desired class.
Let $\widehat{BP}=L_{K(2)}BP$. Again we have an equivariant version $\widehat{\B\bbP}$.
We write $\B\U$ for the $\Z/2$--equivariant space of finite-dimensional real subspaces of $\CC^\infty$ with the complex conjugation as involution. 
By \cite[2.25, 2.28]{MR1808224}  the periodicity class $P$ as defined in Convention \ref{convention} has an equivariant lift $\iota(P)\in \B \bbP^\star$ and 
there is a canonical isomorphism
$ \left(\iota(P)^{-1}\widehat{ \B \bbP} \right)^{2i(1+\alpha)}\B\U  \cong  \left(P^{-1}\widehat{ BP}\right)^{4i}BU.$
Since $P^{-1}\widehat{BP}\to \widehat E$ is surjective on coefficients, we have a surjection $\left(P^{-1}\widehat{BP}\right)^*BU\to \widehat E^*BU$.
The map $\widehat E^*BU\to \widehat E^*BSpin$ is surjective by Lemma \ref{surjbubspin}.

The top right diagonal map in the diagram is restriction to fixed points $\widehat \E^*\B\U \lra \left(\widehat\E^{\Z/2}\right)^*BO$ composed with restriction from $BO$ to $BSpin$. 
Note that restriction to fixed points is only possible for integral degree equivariant cohomology. 
This is why we proceed as in \cite[Section 5]{KW13} and multiply with powers of the invertible class $y\in \widehat\E^{-17-\alpha}$ which non-equivariantly corresponds to the class $v_2^3$.
\end{proof}

\begin{proof}[Proof of \ref{main Spin}. ]
We equip $BSpin$ with the trivial $\Z/2$--action. Since the spectra  $ \TM$ and $\widehat{\TM}$  satisfy the completion theorem,
so do the function spectra $F(BSpin,\TM)$ and $F(BSpin,\widehat{\TM})$. It follows that we have two Borel spectral sequences and a comparison map 
between them. Also, both spectral sequences for the function spectra are modules over the spectral sequences for a point.
The integral degree parts of the spectral sequences are the homotopy fixed point spectral sequences 
$$E_1^{s,t}=C^s(\Z/2, TMF_1(3)^{t}BSpin)\Rightarrow (TMF_1(3)^{h\Z/2})^{s+t}BSpin$$
and 
$$E_1^{s,t}=C^s\left(\Z/2, \widehat{TMF_1(3)}^{t}BSpin\right)\Rightarrow \left(\widehat{TMF_1(3)}^{h\Z/2}\right)^{s+t}BSpin$$
respectively. By \cite{L13} we have $$TMF_1(3)^*BSpin\cong TMF_1(3)^*[[\pi _1,\pi _2,\dots]]$$ 
and $$\widehat{TMF_1(3)}^*BSpin\cong \widehat{TMF_1(3)}^*[[\pi _1,\pi _2,\dots]].$$ 
We first consider the second Borel spectral sequence:
Its $E_1$--term is $$E_1=\left(\widehat{TMF_1(3)}^*BSpin\right)[\sigma^{\pm 1},a].$$
From Proposition \ref{equivariant Pontryagin}  we know that that the Pontryagin classes $\pi_i$ are invariant --- a more geometric proof of this property is given in Appendix \ref{GeomInvariance} --- and survive the homotopy fixed point spectral sequence. 
Thus the Pontryagin classes are permanent cycles in the Borel spectral sequence, and the integral degree part converges to 
$$ \widehat{TMF_0(3)}^*[\![\pi _1,\pi _2,\ldots]\!]. $$

By induction
the comparison map between the two Borel spectral sequences is injective on all $E_r$--terms. It follows that the $\pi_i$ are also permanent cycles in the first Borel spectral sequence, so that the integral degree part converges to
$$ TMF_0(3)^*[\![\pi _1,\pi _2,\ldots]\!]. $$
 and Theorem \ref{main Spin} is proven. 
 \end{proof}
From the Borel spectral sequence it follows that $TMF_0(3)^{32j}\to TMF_1(3)^{32j}$ is injective for all $j$, see also \cite[page 1011]{MR3044602}. 
This implies that $TMF_1(3)^{-32i}BSpin\to TMF_0(3)^{-32i}BSpin$ is injective and the lift
of $\pi_i\in TMF_1(3)^{-32i}BSpin$ to $\pi_i\in TMF_0(3)^{-32i}BSpin$ is uniquely determined.
  \begin{cor}
The total Pontryagin classes $\pi_t=1 +\pi_1 t+\pi_2 t^2 +\ldots$ satisfy the Cartan formula, i.e. for all spin bundles $\xi$ and $\eta$ over a space $X$ we have
  $$\pi _t(\xi \oplus \eta) = \pi _t(\xi)\pi _t(\eta) \in (TMF_0(3)^*X)[\![t]\!].$$
 \end{cor}
\begin{proof}
This property of the $TMF_0(3)$--Pontryagin classes is inherited from the $TMF_1(3)$--Pontryagin classes.
\end{proof}
\section{Cubical structures and the difference class}\label{CubicalStr}
After the construction of the Pontryagin classes we now consider the remaining generator $r$ of the $TMF_1(3)$--cohomology of $BString$. In \cite[Theroem 1.3]{L13} this class has been specified by the property that it maps to a generator under the canonical map to the cohomology of $K(\Z,3)$. However, not all of these choices will have a lift to $TMF_0(3)$.
In this section we will give a specific  choice of the class which allows the desired lift. It measures the difference of  two specific orientations of $TMF_1(3)$.
\par
We start by recalling the construction of the $\sigma$--orientation $\sigma:MU\langle 6\rangle \to E$ for any elliptic spectrum $E$ 
using cubical structures from \cite[Section 2]{MR1869850}. 

Let $C$ be an elliptic curve over a base scheme $S$. 
The ideal sheaf ${\mathcal I}(0)={\mathcal I}_S$ associated to the zero section $S\to C$ defines a line bundle over $C$.
For $G$ a (formal) group scheme over $S$ and all subsets $J\subseteq\{0,1,2\}$ there are maps 
\begin{align*}
\sigma_J\colon G^3 =G\times_SG\times_SG &\to G\\
(g_0,g_1,g_2)&\mapsto \sum_{j\in J} g_i.
\end{align*}
For a line bundle $L$ over $G$ we define the line bundle $\Theta^3(L)$ over $G^3$ by 
$$\Theta^3(L)=\bigotimes_{J\subseteq\{0,1,2\}}\left( \sigma_J^*L\right)^{(-1)^{\lvert J \rvert}}.$$
Recall from \cite[Definition 2.40]{MR1869850} that a cubical structure on 
$L$ is a global section of $\Theta^3(L)$ satisfying certain properties.
The theorem of the cube  \cite[Theorem 2.53]{MR1869850} implies that each elliptic curve has a unique cubical structure on ${\mathcal I}(0)$.
On the other hand,  if $\hat{C}$ denotes the formal group of the elliptic curve,
and $E$ is an elliptic cohomology theory such that $\hat{C}\cong {\rm Spf}\, E^0{\mathbb C}P^{\infty}$, then we have 
\begin{thm}\cite[Corollary 2.50]{MR1869850}
Cubical structures on the restriction of ${\mathcal I}(0)$ to $\hat{C}^3$
are in bijection with ring spectrum maps $MU\langle 6 \rangle\to E$.
\end{thm}
In particular, the cubical structure on $C$ defines a distinguished cubical structure $s$ on $\hat{C}$,
and a distinguished ring spectrum map $\sigma:MU\langle 6 \rangle\to E$.

If $C\to S$ is given in Weierstrass form
$$Y^2Z+a_1XYZ+a_3YZ^2=X^3+a_2X^2Z+a_4XZ^2+a_6Z^3,$$ 
that is, as the subset of ${\mathbb P}^2$ of all $c=[X:Y:Z]$ satisfying the above equation, 
the cubical structure on ${\mathcal I}(0)$ is given by the section $s(c_0,c_1,c_2)=t(c_0,c_1,c_2)d(X/Y)_0$ of $\Theta^3({\mathcal I}(0))$ over $C^3$,
where $$t(c_0,c_1,c_2)
=\frac{\left\lvert\begin{matrix}X_0&Z_0\\ X_1 & Z_1\end{matrix}\right\rvert \left\lvert\begin{matrix}X_1&Z_1\\ X_2 & Z_2\end{matrix}\right\rvert \left\lvert\begin{matrix}X_2&Z_2\\ X_0 & Z_0\end{matrix}\right\rvert}{\left\lvert\begin{matrix}X_0&Y_0&Z_0\\ X_1 & Y_1 & Z_1\\ X_2 & Y_2 & Z_2\end{matrix}\right\rvert Z_0Z_1Z_2}.$$
Here $c_i=[X_i:Y_i:Z_i]$, and $d(X/Y)_0$ is a section of the bundle $p^*\omega$, where $p:C^3\to S$ is the projection. Note that $t$ is a function on $C^3$ with divisor
$$D=-\sum_i [c_i=0]+\sum_{i<j} [c_i+c_j=0]-[c_0+c_1+c_2=0],$$ 
where $[c_i=0]$ denotes the largest closed subscheme of $C^3$ where $c_i$ is equal to the zero section of the elliptic curve, and so on.
The function $t$ is a trivialization of the corresponding line bundle ${\mathcal I}_D$, and one has
an isomorphism $$\Theta^3({\mathcal I}(0))\cong {\mathcal I}_D\otimes p^*\omega.$$
See \cite[Appendix B.4.3]{MR1869850}.

We want to use the Thom isomorphism $E^*MU\langle 6 \rangle \cong E^*BU\langle 6\rangle$ to obtain a class in $E^*BU\langle 6\rangle$, therefore we need to choose a Thom class.
Each complex orientation of $E$ induces a Thom class, and defining a complex orientation of $E$ 
is equivalent to defining a coordinate on the formal group $\hat{C}$.
It thus makes sense to use the same notation $x$ for  a coordinate on the formal group $\hat{C}$ and an element of $E^*BU(1)$ and a ring spectrum map $MU\to E$.
By abuse of notation, we also denote the composition $MU\langle 6\rangle \to MU \to E$ by $x$.

The coordinate $x$ defines a trivialization of the bundle ${\mathcal I}(0)$. 
Using this trivialization, cubical structures on the restriction of ${\mathcal I}(0)$ to $\hat{C}^3$ correspond to cubical structures
on the trivial line bundle over $\hat{C}^3$ (or equivalently power series in three variables satisfying certain properties). 
Ando, Hopkins and Strickland show that the latter can be identified
with ring spectrum maps $BU\langle 6 \rangle_+\to E$, where we have identified the space $BU\langle 6 \rangle_+$ with its suspension spectrum.
This uses composition with the map 
 \begin{eqnarray}\label{map f}
 &&f=\prod(1-L_i): (\CC\bP^\infty)^3\to BU\langle 6 \rangle
 \end{eqnarray}
which is the unique lift up to homotopy of a corresponding map $$\prod(1-L_i): (\CC\bP^\infty)^3\to BU$$ which classifies the exterior tensor product
of the virtual bundle $1-L$ over each factor. Here $L$ is the canonical line bundle on $\CC\bP^\infty$.
Note the  isomorphism $E^*(\CC\bP^\infty)^3\cong E^*[[x_0,x_1,x_2]]$ after a choice of coordinate.

Now we consider the formal group of the Weierstrass curve:
the zero section of $C$ is $[X:Y:Z]=[0:1:0]$, and we choose
the function $x=X/Y$ as a coordinate on $\hat{C}$. We also define the function $z=Z/Y$. 
These new ``variables'' $x$ and $z$ are coordinates for the affine chart $Y=1$ of ${\mathbb P}^2$, which contains the zero section (and $\hat{C}$).
The intersection of $C$ with this affine chart consists of all points with coordinates $(x,z)$ such that
$$z+a_1xz+a_3z^2=x^3+a_2x^2z+a_4xz^2+a_6z^3.$$ 
The restriction of $z$ to $\hat{C}$ is a formal power series $$z(x)=x^3-a_1x^4+(a_1^2+a_2)x^5-(a_1^3+2a_1a_2+a_3)x^6+\dots$$
in the coordinate $x$, which is obtained by solving the previous equation for $z$. One also has a formal expansion of the addition $+_F$ on $\hat{C}$:
$$x_0+_F x_1=x_0+x_1+a_1x_0x_1-a_2(x_0^2x_1+x_0x_1^2) + \dots$$
The trivialization $x$ of the restriction of ${\mathcal I}(0)$ leads to corresponding trivializations $d(X/Y)_0$ of the restriction of $p^*\omega$
and $$u(x_0,x_1,x_2)=\frac{(x_0+_F x_1)(x_1+_F x_2)(x_2+_F x_0)}{x_0x_1x_2(x_0+_F x_1+_Fx_2)}$$
of the restriction of ${\mathcal I}_D$.
Also $$t=\frac{\left\lvert\begin{matrix}x_0&z_0\\ x_1 & z_1\end{matrix}\right\rvert \left\lvert\begin{matrix}x_1&z_1\\ x_2 & z_2\end{matrix}\right\rvert \left\lvert\begin{matrix}x_2&z_2\\ x_0 & z_0\end{matrix}\right\rvert}{\left\lvert\begin{matrix}x_0&1&z_0\\ x_1 & 1 & z_1\\ x_2 & 1 & z_2\end{matrix}\right\rvert z_0z_1z_2}$$
becomes a section of the trivial bundle, hence a quotient of two formal power series in $x_0,x_1,x_2$, since $z_i=z(x_i).$

Finally the power series $\frac tu$ is the cubical structure on the trivial line bundle which corresponds to $\sigma$ under the Thom isomorphism with Thom class $x$.
Since the Thom isomorphism $E^*BU\langle 6 \rangle \to  E^*MU\langle 6\rangle$ is given by multiplying with $x$, we denote the cohomology class corresponding to 
the inverse image of $\sigma$ by $$r_U=\frac{\sigma}x.$$ In the literature such a class is called a difference class of the two Thom classes $\sigma$ and $x$.
The corresponding reduced class $\frac{\sigma}{x}-1$ is denoted by $\tilde{r}_U$ in the sequel. 
\begin{prop}\label{cubical expansion}
The cubical structure  $r_U=\frac{t(x_0,x_1,x_2)}{u(x_0,x_1,x_2)}\in E^*[[x_0,x_1,x_2]]$ 
is given by
$$ 1-(a_1a_2-3a_3)x_0x_1x_2-(a_1a_3-a_2^2+5a_4)(x_0^2x_1x_2+x_0x_1^2x_2+x_0x_1x_2^2)+\dots.$$
\end{prop}
This can be calculated with a computer algebra system. However, 
later we will only use that the coefficient in front of $x_0x_1x_2$ equals $a_3$ when reduced modulo $(2,a_1,a_2)$.
This can easily be calculated by hand, since modulo $(2,a_1,a_2)$ the formal group law  and the series $z$ read 
\begin{align*}
x_0+_Fx_1 &= x_0+x_1+a_3x_0^2x_1^2 +\text{higher terms},\\
z(x)&= x^3+a_3x^6 + \text{higher terms}.
\end{align*}
See Appendix \ref{Comp} for details of this computation.

In the following we consider the elliptic spectrum $TMF_1(3)$ whose corresponding elliptic curve 
has a canonical Weierstrass form. Therefore we obtain a distinguished class $r_{U}=\frac{\sigma}{x}\in TMF_1(3)^*BU\langle 6 \rangle$.

There is a commutative diagram of fibrations:
\begin{eqnarray}\label{cpj}
&\xymatrix{
K(\Z,3)
\ar[r]^i\ar[d]^{\cdot 2}
&
BString 
\ar[r]\ar[d]^c
&
BSpin
\ar[d]
\\
K(\Z,3)
\ar[r]^j
&
BU\langle 6 \rangle
\ar[r]
&
BSU
}&
\end{eqnarray}
where $c$ is induced from complexification of vector bundles.
We denote $$r=c^*r_U\in TMF_1(3)^*BString.$$

\begin{remark}
It is interesting to note that since $\sigma: MU\langle 6\rangle \to TMF_1(3)$ factors \cite{AHR10} through the realification map $MU\langle 6\rangle\to MString$,
both complexification and realification appear in the definition of the class $r$.
However, the class $r$ does not factor through the map $BString\to BString$
which is given by the composition of complexification and realification --- this would be multiplication by 2
in the H--space $BString$, and so cannot produce a generator in cohomology.
\end{remark}

Next we will show that the reduced class $\tilde{r}=r-1$ defines a generator in the $TMF_1(3)$--cohomology of $BString$. 
\begin{thm}\label{EBString}
Let $\pi_i$ be the Pontryagin classes of Proposition \ref{equivariant Pontryagin}. Then the map 
$$ TMF_1(3)^*[[ \tilde{r}, \pi _1,\pi _2,\dots]]\longrightarrow TMF_1(3)^*BString$$
which sends the formal variables in the source to classes with the same names in the target is injective. 
Moreover,  it is  an isomorphism for the completed cohomology theory:
$$\widehat{TMF_1(3)}^*[[ \tilde{r}, \pi _1,\pi _2,\dots]]\cong \widehat{TMF_1(3)}^*(BString).$$

\end{thm}

Before proving the theorem consider the following commutative diagram, where everything is localized at the prime 2:
\begin{eqnarray}\label{cpbp1bu6}
&\xymatrix{
&(\CC\bP^\infty)^3
\ar[d]
& \\
K(\Z,3)
\ar[r]\ar[d]^=
&
 \underline{BP\langle 1\rangle}_6 
\ar[r]\ar[d]
&
 \underline{BP\langle 1\rangle}_4 
\ar[d]
\\
K(\Z,3)
\ar[r]^j
&
\underline{ku}_6=BU\langle 6 \rangle
\ar[r]
&
 \underline{ku}_4=BSU.
}
\end{eqnarray}
Here, the spaces in an $\Omega$--spectrum $E$ are denoted by $\underline{E}_n$.
The rows of the diagram are the obvious fibrations which arise from multiplication by $v_1$ on the spectra $BP\langle 1\rangle$ and $ku$.
The  inclusion $\CC\bP^\infty \to MU(1)\to \underline{MU}_2$ 
and the  ring  map $MU\to BP\to BP\langle 1 \rangle$ induce the complex coordinate 
$$\CC\bP^\infty \to \underline{BP\langle 1\rangle}_2.$$
The upper vertical map in the diagram is obtained by composing the threefold product of this coordinate with the multiplication maps
for the spaces $\underline{BP\langle 1\rangle}_*$.

The spectrum $ku$ is complex connective $K$--theory. It is complex oriented by $1-L$ with the canonical line bundle $L$ over $\CC\bP^\infty$. 
The $2$--typicalization of this orientation  factors through $MU\to BP\langle 1 \rangle$ which gives the lower vertical maps of the diagram. Note that the composition of the two vertical maps in the middle column coincide with the 
 2--typicalization of the map $f$ 
  which has been used in Equation \eqref{map f} for the identification with the cubical structures. The map and its typicalization coincide on the product $(\CC P^1)^3$ since they are products of Euler classes.

The $K(2)$--homology of $\CC\bP^\infty$ is a free module on canonical generators $\beta_i\in K(2)_*\CC\bP^\infty$ for $i\ge 0$, 
which are dual to the elements $x^i\in K(2)^*\CC\bP^\infty$.
We denote the image of $\beta_i$ in $K(2)_*\left(\underline{BP\langle 1\rangle}_2\right)$  by $b_i$, and we abbreviate $b_{(i)}=b_{2^i}$. With this notation we have just proved that the image of the class $\beta_1^{\otimes 3}\in K(2)_*(\CC\bP^\infty)^3$  under the vertical composite coincides with its image under the map $f$.

\begin{lemma}\label{sulemma} A class in $K(2)^*\underline{BP\langle 1\rangle}_6$ maps to a generator of $K(2)^*K(\Z,3)$
if its restriction to $(\CC\bP^\infty)^3$ pairs non-trivially with $\beta_1^{\otimes 3}\in K(2)_*(\CC\bP^\infty)^3$.
\end{lemma}

The following proof is based on work by Ravenel--Wilson \cite{MR0448337} and Ravenel--Wilson--Yagita \cite{MR1648284} and is described in Section 8.5 of the latter. It appears in similar form in Su's PhD thesis  \cite{MR2709571}.

\begin{proof}
For a ring spectrum there are addition maps $$\underline{E}_m\times \underline{E}_m\to \underline{E}_{m}$$ and multiplication 
maps $$\underline{E}_m\times \underline{E}_n\to \underline{E}_{m+n}.$$ The direct sum $K(2)_*(\underline{E}_*)$ over all homological degrees 
and all spaces in the spectrum is thus equipped with three operations: the usual addition $+$, 
an operation $*$ induced by the addition maps for the $\underline{E}_*$,
and an operation $\circ$ induced by the multiplication maps for the $\underline{E}_*$.
Ravenel--Wilson call such a structure a Hopf ring. 

The image of $\beta_1^{\otimes 3}$ in $K(2)_*\underline{BP\langle 1\rangle}_6$ is $b_{(0)}^{\circ 3}$, so it suffices to show that a class is a generator
if it pairs non-trivially with $b_{(0)}^{\circ 3}$.

By  \cite{MR1648284} the middle row of Diagram \eqref{cpbp1bu6} gives rise to a short exact sequence of bicommutative 
Hopf algebras $$K(2)^*( \underline{BP\langle 1\rangle}_4 ) \to K(2)^*( \underline{BP\langle 1\rangle}_6 ) \to K(2)^*(K(\Z,3) ) ,$$
and all three algebras are isomorphic to power series algebras on the indecomposables.
We want to compute the first map on indecomposables. We dualize and compute the induced map in homology 
on primitives. 

The homotopy of the spectrum $BP\langle 1 \rangle$ contains elements $v_1^k$ for $k\ge 0$.
Using the Hurewicz map, we have corresponding elements $[v_1^k]=[v_1]^{\circ k}\in K(2)_0(\underline{BP\langle 1\rangle}_{-2k})$. 
The map $K(2)_*( \underline{BP\langle 1\rangle}_6 ) \to K(2)_*( \underline{BP\langle 1\rangle}_4 )$ is $\circ$--multiplication with $[v_1]$.

From \cite[Theorem 5.3]{MR0448337}, it follows that the primitives of $K(2)_*( \underline{BP\langle 1\rangle}_4 )$ form the free $K(2)_*$--module with basis elements
$$[v_1^{k-1}]\circ b_{(j_1)}\circ b_{(j_2)} \circ \dots \circ b_{(j_k)} \circ b_{(0)},\qquad\text{where $k\ge 1$ and $0\le j_1 < j_2< \dots < j_k$.}$$ 

The primitives of $K(2)_*( \underline{BP\langle 1\rangle}_6 )$ form the free $K(2)_*$--module with basis elements
$$[v_1^{k-2}]\circ b_{(j_1)}\circ b_{(j_2)} \circ \dots \circ b_{(j_k)} \circ b_{(0)},\qquad
\text{where $k\ge 2$ and $0\le j_1 < j_2< \dots < j_k$,}$$
and 
$b_{(i)}^{\circ 2}\circ b_{(0)}$ for $i\ge 0$.

The spectrum $K(2)\wedge BP\langle 1 \rangle$ has two formal group laws, one inherited from each factor,
and a canonical isomorphism between them. We need to see what happens with this structure when we replace the spectrum $BP\langle 1 \rangle$ by 
the spaces $\underline{BP\langle 1\rangle}_*$.

We extend the operations $\circ $ and $*$ to power series over the Hopf ring $K(2)_*(\underline{BP\langle 1\rangle}_*)$ in a variable $t$ by defining 
\begin{align*}
at^i\circ bt^j &=(a\circ b)t^{i+j},\\ 
at^i * bt^i &= (a*b)t^i\text{, and }\\
at^i * bt^j&=at^i+ bt^j\text{ for $i\ne j$}
\end{align*}
for $a,b\in K(2)_*(\underline{BP\langle 1\rangle}_*)$.
Similarly for power series rings in several variables.

We can consider elements of $K(2)_*$ as elements of our Hopf ring $K(2)_*(\underline{BP\langle 1\rangle}_*)$
by multiplying them with $1\in K(2)_0(\underline{BP\langle 1\rangle}_0)$. Hence
the formal group law from $K(2)$ gives a power series $$s+_{K(2)}t\in K(2)_*(\underline{BP\langle 1\rangle}_*)[[s,t]].$$ 
The second formal group law $$s+_{BP\langle 1 \rangle} t =\sum_{i,j}a_{ij}s^it^j\in BP\langle 1\rangle_*[[s,t]]$$ induces an operation 
$$s+_{[BP\langle 1\rangle]}t=\mfgl\limits_{i,j\ge 0} [a_{ij}] \circ s^{\circ i}t^{\circ j}$$
on power series over our Hopf ring.
Finally the isomorphism gives a series $b(s)=\sum b_is^i$ which satisfies the equality \cite[Theorem 3.9]{MR0448337}: 
$$b(s+_{K(2)}t)=b(s)+_{[BP\langle 1\rangle]}b(t).$$
Setting $s=t$ in the last equation, the well-known two--series of the group laws for $K(2)$ and $BP\langle 1 \rangle$ give us
$\sum_q b_qv_2^qt^{4q}=  [v_1]b(t)^{\circ 2}$ and so $v_2^{2^{i}}b_{(i)}=[v_1]b_{(i+1)}^{\circ 2}$ for all $i\ge 0$, as well as $0=[v_1]b_{(0)}^{\circ 2}$.

We can now see explicitly the map $K(2)_*( \underline{BP\langle 1\rangle}_6 ) \to K(2)_*( \underline{BP\langle 1\rangle}_4 )$ 
on the generators for the primitives:
The generators
$[v_1^{k-2}]\circ b_{(j_1)}\circ b_{(j_2)} \circ \dots \circ b_{(j_k)} \circ b_{(0)}$ are mapped to the generators  
$[v_1^{k-1}]\circ b_{(j_1)}\circ b_{(j_2)} \circ \dots \circ b_{(j_k)} \circ b_{(0)}$ with $k\ge 2$.
The generators $b_{(i+1)}^{\circ 2}\circ b_{(0)}$ are mapped to $v_2^{2^{i}}b_{(i)} b_{(0)}$, i.e. a product of an invertible element and of one of the remaining generators
$[v_1^{k-1}]\circ b_{(j_1)}\circ b_{(j_2)} \circ \dots \circ b_{(j_k)} \circ b_{(0)}$ with $k=1$.
Finally the last generator $b_{(0)}^{\circ 3}$ is mapped to zero.

By dualizing, we see that a class in $K(2)^*( \underline{BP\langle 1\rangle}_6 )$ is mapped to a generator of the indecomposables for $K(2)^*K(\Z,3)$
if it pairs non-trivially with the primitive $b_{(0)}^{\circ 3}$. \end{proof}

\begin{proof}[Proof of Theorem \ref{EBString}]
We denote the reduction of $r$  by $(2,v_1)$  by the same name ${r}\in K(2)^*BString$.
By \cite[Theorem 1.5]{MR2093483} there is an epimorphism of Hopf algebras 
$$p:K(2)^*BString\to K(2)^*K(\Z,3)$$
which arises from the diagram $$K(2)^*BString\stackrel{i^*}\to K(2)^*K(\Z,3)\stackrel{ (\cdot 2)^*}\leftarrow K(2)^*K(\Z,3)$$
as both maps have the same image and the right map is a monomorphism.
Since ${r}=c^*{r}_U$, we have with Diagram \eqref{cpj} $$p({r})=j^*{r}_U.$$

In \cite[Section 12]{MR584466} Ravenel and Wilson prove that the Hopf algebra $K(2)_*K(\Z,3)$ is a divided power algebra, 
and the dual Hopf algebra $K(2)^*K(\Z,3)$ is a power series algebra on one generator. 
By \cite[Theorem 1.3]{L13} it is enough to show that $p(\tilde{r})$ is a free topological generator for 
$K(2)^*K(\Z,3)$, that is, $$K(2)^*K(\Z,3)\cong K(2)^*[[p(\tilde{r})]].$$

The image of $r_U$ in $K(2)^*(\CC\bP^\infty)^3$ has an invertible coefficient in front of $x_0x_1x_2$ by Proposition \ref{cubical expansion}.
Hence Lemma \ref{sulemma} shows the claim about the $\widehat{TMF_1(3)}$--cohomology of $BString$. 

For the $TMF_1(3)$--cohomology we see that the composition
$$ TMF_1(3)^*[[ \tilde{r}, \pi _1,\pi _2,\dots]]\longrightarrow TMF_1(3)^*BString\to \widehat{TMF_1(3)}^*BString$$
is injective, so the first map must be injective.
\end{proof}

\section{The $TMF_0(3)$--cohomology of $BString$} 

In this section we show that the class $r_U\in TMF_1(3)^*BU\langle 6\rangle$ lifts to an equivariant cohomology class.
As we will see below, the involution on the classifying space $BU$ induces an involution on $BU\langle 6\rangle$, and we denote the resulting $\Z/2$--space by $\B\U\langle 6\rangle$. We show that $r_U$ lifts to an equivariant map $\B\U \langle 6\rangle\to \TM$. 
Restricting this map to the fixed points in both source and target, we conclude that the class ${r}$ comes from a class in $TMF_0(3)^*BString$.
This finally allows us to prove Theorem  \ref{main}.

To make the above induced involution on $BU\langle 6\rangle$ precise consider the first stages of the equivariant Whitehead tower of $\B\U$. 
For a Mackey functor $M$, write $K(M,V)$ for the Eilenberg--MacLane space, that is, the $V$th space in the $\Omega$--spectrum of $HM$ (see \cite[2.4]{MR2240234} for details).
Let $\B\U \langle 4\rangle= \B\S\U$ be the homotopy fibre of the first Real Chern class  (ibid.\ Section 5)
$$c_1: \B\U \to K(\underline{\Z}, 1+\alpha),$$ let $\B\U\langle 6\rangle$ be the fibre of 
$$c_2: \B\U \langle 4\rangle \to K(\underline{\Z}, 2(1+\alpha)),$$ and let $\B\U\langle 8\rangle $ be the fibre of
$$c_3: \B\U \langle 6\rangle \to K(\underline{\Z}, 3(1+\alpha)).$$

Note that non-equivariantly $\B\U \langle 6\rangle $ coincides with $BU \langle 6\rangle $ and the first five homotopy groups vanish. 
After taking $\Z/2$--fixed points, the spaces $\B\U\langle 2k\rangle^{\Z/2}$ are $(k-1)$--connected (see ibid.\ Proposition 3.6(a)). We have $\B\U^{\Z/2}=BO$, $\B\S\U^{\Z/2}=BSO$ and fibrations (see ibid.\ Corollary 2.12)
\begin{align*}
\B\S\U^{\Z/2} &\to&  \B\U^{\Z/2} &\to K(\underline{\Z}, 1+\alpha)^{\Z/2}=K(\Z/2,1)\\
\B\U \langle 6\rangle^{\Z/2}& \to& \B\S\U^{\Z/2}& \to K(\underline{\Z}, 2(1+\alpha))^{\Z/2}=K(\Z/2,2)\times K(\Z,4)\\
\B\U \langle 8\rangle^{\Z/2} &\to& \B\U \langle 6\rangle^{\Z/2} &\to K(\underline{\Z}, 3(1+\alpha))^{\Z/2}=K(\Z/2,3)\times K(\Z/2,5).
\end{align*}
\begin{lemma}
 $\B\U \langle 6\rangle^{\Z/2}$ is the homotopy fiber of $p_1: BSpin\to K(\Z,4)$.
 \end{lemma}
 \begin{proof}
Since $\B\U \langle 6\rangle^{\Z/2}$ is 2--connected, the map $BSO\to K(\Z/2,2)$ in the second fibration is the second Stiefel--Whitney class $w_2$, and since 
in the third line $\B\U \langle 6\rangle^{\Z/2}\to K(\Z/2,3)$ is an isomorphism on $\pi_3$, the map $BSO\to K(\Z,4)$ in the second line is $p_1$.
\end{proof}
It follows that there is a map which fits into the diagram
$$\xymatrix{BString \ar@{-->}[r]\ar[d] &  \B\U \langle 6\rangle^{\Z/2}\ar[r]\ar[d]& BU \langle 6\rangle\ar[d]\\ BSO\ar[r]^\cong& \B\S\U ^{\Z/2}\ar[r]& BSU.}$$
Since the composite of the lower horizontal maps is the complexification map so is the upper one. Hence, if we equip $BString$ with the trivial $\Z/2$--action, we have an equivariant map $$BString\to \B\U \langle 6\rangle$$ which non-equivariantly is the complexification.

\begin{lemma} \label{MUBP}
The forgetful map
\begin{eqnarray*}
\E ^{k(1+\alpha)} \B \U \langle 6\rangle & \stackrel{}{\lra} &E^{2k}BU  \langle 6\rangle
\end{eqnarray*}
is an isomorphism for every $\bbM \U$--module spectrum $\E$ such that $\E\cong c{\E }$ and $\E ^{*(1+\alpha)}\cong E^{2*}$. In particular this applies to $\E=\TM$.
\end{lemma}
\begin{proof} 
Since the spectrum $\E$ is suitably complete, we can replace it with $c{\E }=F(E\Z/2_+,\E )$. 
Now the argument follows \cite[Theorem 2.3]{MR3044602} and we repeat it here: set $X= (\CC P^\infty)^{3}$, $Z=\Omega \Sigma X$ and $Y=\B\U \langle 6\rangle$. 
Note that $Y$ is a space in the Real connective $K$-theory spectrum, so there exists a $\Z/2$-space $Y'$ such that $Y=\Omega Y'$ equivariantly. 
Recall the map $f:X\to Y$ from Equation \eqref{map f}.
It is equivariant and the ordinary homology of $Y$ is generated as an algebra by its image, see \cite[Theorem 2.9]{ MR1869850}. Thus the equivariant map from $Z$ to $Y$ obtained from $f$ is surjective in ordinary homology and, by the Atiyah--Hirzebruch spectral sequence, split surjective in $MU$--homology. 

The $\Z/2$--space $Z$ admits a James filtration and hence equivariantly splits into a wedge of spectra of the form $X^k$ \cite[Theorem 6.2]{MR2607411}.
The latter space has an equivariant cellular decomposition into Real cells, i.e. disks in some $\mathbb C^n$ with complex conjugation as the involution. The
spectral sequence for ${\bbM \U}_\star Z$ corresponding to this cellular filtration collapses, and we see that 
$${\bbM \U} \wedge Z \cong \bigvee \Sigma^{k_i(1+\alpha)}{\bbM \U} .$$
Since we may assume that the connectivities of the summands increase it is a free ${\bbM \U}$--module of finite type.
Choose a subsequence $\beta_1,\beta_2,\ldots $ of $k_1, k_2,\ldots$ so as to obtain an equivariant map
$$ \bigvee \Sigma^{\beta_i(1+\alpha)} {\bbM \U} \longrightarrow Y\wedge {\bbM \U} $$
which is a non-equivariant homotopy equivalence. This gives an equivariant equivalence of ${\bbM \U} $--module spectra
$$ \bigvee \Sigma^{\beta_i(1+\alpha)} E\Z /2_+\wedge {\bbM \U} \longrightarrow E\Z /2_+\wedge Y\wedge {\bbM \U} .$$
Let $F$ (and $F_{\bbM \U}$) be the equivariant function spectrum (resp.\ the ${\bbM \U}$--module function spectrum). Then we have the equivalences   
\begin{eqnarray*}
\E ^{\star}Y& \cong &c{\E }^{\star}Y \cong \E ^{\star}(Y\wedge E\Z/2_+)\\
& \cong & 
F(Y\wedge E\Z/2_+, \E )^\star \cong   F_{\bbM \U}(Y\wedge E\Z/2_+ \wedge {\bbM \U} , \E )^\star \\
& \cong & 
F_{\bbM \U}\left( \bigvee \Sigma^{\beta_i(1+\alpha)} E\Z /2_+\wedge {\bbM \U} , \E \right)^\star \cong \prod 
F_{\bbM \U}( E\Z /2_+\wedge {\bbM \U} , \E )^{\star -\beta_i(1+\alpha)}
\\
&\cong& \prod 
F ( E\Z /2_+ , \E )^{\star -\beta_i(1+\alpha)} \cong  \prod \E ^{\star -\beta_i(1+\alpha)} \cong
 \E ^{\star}\langle \!\langle \gamma_1, \gamma_2, \ldots \rangle \! \rangle 
\end{eqnarray*}
with generators $\gamma_i$ of degree $\beta_i(1+\alpha)$ which also freely generate the non-equivariant cohomology. 

The fact that $\TM ^{*(1+\alpha)}\cong TMF_1(3)^{2*}$ follows easily from our description of the Borel spectral sequence below Lemma \ref{SCT}.
\end{proof}

\begin{proof}[Proof of Theorem \ref{main}. ]
Set $E=TMF_1(3)$. Consider the function $\Z/2$--spectra $F(?, \E)$ whose fixed points
$F(?, \E)^{\Z/2}$ are the spectra of equivariant functions. For the trivial $\Z/2$--space $BString$, the latter is equivalent to the function spectrum $F(BString, \E^{\Z/2})$ over the trivial universe.
By Lemma \ref{SCT} we have a strong completion theorem for $\E$ and the homotopy fixed point spectral sequence converges to the homotopy of the fixed point spectrum:
\begin{align*} 
E_2&=H^*(\Z/2,E^*(BString)) &\;\Rightarrow \quad & \pi_* F(BString, \E)^{\Z/2} \cong (\E^{\Z/2})^*(BString)\end{align*}
The class
 $r_U:BU\langle 6\rangle \to E$ is represented by an equivariant map $B\U\langle 6\rangle \to \E$ by Lemma \ref{MUBP}.  Taking fixed points   and composing with the map from $BString$  we have a map  $r:BString\to \E^{\Z/2}$.
It follows that $r\in E_2^{0,*}$ is a permanent cycle. 
The multiplicativity of $r$ under direct sums follows from the multiplicativity of the Thom classes $\sigma$ and $x$ respectively.

For the calculation of the $\hat{E}^{h\Z/2}=\widehat{TMF_0(3)}$--cohomology of $BString$, we can employ the homotopy fixed point spectral sequence
$$E_1^{s,t}=C^s(\Z/2, \hat{E}^{t}BString)\Rightarrow (\hat{E}^{h\Z/2})^*BString.$$
We have $$\hat{E}^*BString\cong\hat{E}^*[[\tilde{r},\pi _1,\pi _2,\dots]],$$ and $\Z/2$ acts by $a_i\mapsto -a_i, {r}\mapsto {r}, \pi _i\mapsto \pi _i$.

We know that $\tilde{r}$ is a permanent cycle and so are the Pontryagin classes (by the naturality of the spectral sequence). Hence the spectral sequence converges to 
$$ \hat{E}^{h\Z/2}{}^*[\![ \tilde{r}, \pi _1,\pi _2,\ldots]\!] $$
 and the second part of  Theorem \ref{main} is proven. 
\par
It remains to show the displayed formula for the character of the characteristic class $r$. 
The elliptic character at the cusp $\infty$ is the composite 
$$TMF_0(3)\to TMF_1(3)\to K[\zeta_3] (\!(q)\!).$$
It coincides on homotopy groups with the $q$--expansion map at $\infty$ for modular forms. 

It is known by  \cite[Appendix I, 5.3 and 6.4]{MR1189136} that 
for bundles with vanishing first Pontryagin class, we may replace (in terms of $q$--expanded characters with $q=e^{2\pi i \tau}$) the function $\sigma(\tau,z)$ by the Weierstrass $\Phi$--function $\Phi(\tau,z)$.
Moreover, the complex orientation $MU\langle 6\rangle \to MU \to E$
corresponds to the elliptic genus of level $\Gamma_1(3)$, that is, to the function $$x(z)=\frac{\Phi(\tau,z)\Phi(\tau,-\omega)}{\Phi(\tau,z-\omega)}$$
where $\omega=2\pi i /3$.
Hence, we have the commutative diagram
$$\xymatrix{ BString\ar[d]^r\ar[r]& BU\langle 6\rangle\ar[d] ^{r_U=\frac\sigma x}\ar[r]& BU \ar[d]^\psi \\
E^{h\Z/2} \ar[r]&E\ar[r]& H\CC[v^\pm] (\!(q)\!)^*}
$$
for which  the last vertical arrow with target  the Eilenberg--MacLane spectrum with coefficients  in the periodic ring of complex Laurent series in the formal variable $q=e^{2\pi i\tau}$ is given by the formula
$$\psi(z)=\frac{\Phi(\tau,-\omega)}{\Phi(\tau,z-\omega)}.$$

\end{proof}

\appendix
\section{The invariance of the Pontryagin classes and the difference class} \label{GeomInvariance}
We would like to give a geometric proof for the invariance of the Pontryagin classes  and the difference class.

An isomorphism of Weierstrass curves over a ring $R$ with projective coordinates $[X:Y:Z]$ and $[X':Y':Z']$ respectively is in general given by
$$X'=u^2X+rZ,\quad Y'=su^2X+u^3Y+tZ,\quad Z'=Z,$$ 
where $u\in R^\times, r,s,t\in R$.
We denote by $x$ respectively $x'$ the coordinates on the corresponding formal groups given by the restrictions of the functions $\frac{X}{Y}$ and $\frac{X'}{Y'}$ respectively. On the formal group, the function $z=\frac ZY$ becomes a power series in the coordinate $x$.
This implies $$x'=\frac{u^{-1}x+ru^{-3}z}{1+tu^{-3}z+su^{-1}x}=g(x),$$
where $g$ is a power series in one variable with vanishing constant term and invertible linear coefficient. It is an isomorphism of the formal group laws induced
by the coordinates $x$ and $x'$ respectively. 

In particular we are interested in the universal triple $(C,\omega,P)$ of an elliptic curve
with invariant differential and level structure consisting of a point $P$ of order 3.
This is $$C :\quad Y^2Z+a_1XYZ+a_3YZ^2=X^3$$ over $E^*=TMF_1(3)^*=\Z_{(2)}[a_1,a_3,\Delta^{-1}]$
with $\omega=\frac{dX}{2Y+a_1X+a_3}$ and $P=(0,0)$.

As $(C,\omega,P)$ is the universal triple and $(C,\omega,-P)$ is another such triple, there exists a ring endomorphism of 
$TMF_1(3)^*=\Z_{(2)}[a_1,a_3,\Delta^{-1}]$
such that the pushforward of $(C,\omega,P)$ is isomorphic to $(C,\omega,-P)$.
In fact this endomorphism is an automorphism of order 2: it sends $a_1\mapsto -a_1,a_3\mapsto -a_3$. 

The reason is that $(C,\omega,-P)$ is isomorphic to $(C^-,\omega',P=(0,0))$, where
\begin{align*}
C^-&:\quad  (Y')^2Z'-a_1X'Y'Z'-a_3Y'(Z')^2=(X')^3,\\
\omega'&=\frac{dX'}{2Y'-a_1X'-a_3},
\end{align*} 
via the isomorphism $$X'=X,\quad Y'=a_1X+Y+a_3Z,\quad  Z'=Z.$$
Note that the new coordinates $[X':Y':Z']$ are the coordinates of the negative of the point $[X:Y:Z]$.

The new coordinate on the formal group is $$x'=g(x)=\frac x{1+a_1x+a_3z(x)}.$$ 
But since we have just taken the coordinate of the negative point on the elliptic curve, we also have $$g(x)=\overline{x},$$
since the inverse on the formal group corresponds to the Chern class of the complex conjugate of the canonical bundle over $BU(1)$, i.e. $\overline{x}=[-1](x)$.
Note that this implies $g(g(x))=x$.

\bigskip

We can consider both coordinates $x$ and $x'=\overline{x}$ as ring maps $MU\to TMF_1(3)$.
By composition with $MU\langle 6 \rangle\to MU$, we obtain $x,x': MU\langle 6 \rangle\to TMF_1(3)$, and together with $\sigma:MU\langle 6 \rangle\to TMF_1(3)$,
we obtain $r_U=\frac{\sigma}{x},r_U'=\frac{\sigma}{x'}:BU\langle 6 \rangle_+\to TMF_1(3)$.
Let $\tau$ denote the involutions on $MU\langle 6 \rangle,BU\langle 6 \rangle_+, TMF_1(3)$ respectively.
From our considerations, it follows that the left diagram 
$$
\xymatrix{
MU \ar[d]^\tau \ar[dr]^{x'} \ar[r]^-x
& TMF_1(3) \ar[d]^\tau
&&
MU\langle 6 \rangle \ar[d]^\tau \ar[dr]^{\sigma} \ar[r]^\sigma
& TMF_1(3) \ar[d]^\tau
\\
MU\ar[r]^-x
&TMF_1(3)
&&
MU\langle 6 \rangle\ar[r]^\sigma
&TMF_1(3)
}
$$
commutes up to homotopy: the three maps $MU\to TMF_1(3)$ correspond to the element $x'=g(x)=\overline{x}$.
The fact that the right diagram commutes up to homotopy follows from the construction of $\sigma$ in \cite{MR1869850}.
Therefore $r_U:BU\langle 6 \rangle_+\to TMF_1(3)$ is up to homotopy invariant ---  both compositions
$\tau\circ r_U$ and $r_U\circ\tau$ are homotopic to $r'_U$. 

\bigskip

In the remaining part of this appendix, we consider maximal tori.
We show that the Pontryagin classes are invariant, and we also show (again) that $r_U$ has the same image
in $TMF_1(3)^*BString$ as $r'_U=r_U\cdot \frac{x}{x'}$. (Both $x$ and $x'$ are generators for the free rank one $TMF_1(3)^*BU$--module $TMF_1(3)^*MU$, so that there is a unique invertible element $\frac{x}{x'}$ of $TMF_1(3)^*BU$ whose product with $x'$ is $x$.)

Set $E=TMF_1(3)$. Using the complex coordinate $x$, we have isomorphisms and an injection $$E^*(BU)\cong E^*[[ c_1,c_2,\dots ]] \to E^*(BU(1)^\infty)\cong E^*[[ y_1,y_2,\dots ]],$$
where $BU(1)^\infty=\colim_N BU(1)^N$ and where each $c_k$ is mapped to the $k$--th elementary symmetric polynomial in the $y_i$.
This is induced by the map $$\sum (L_i-1): BU(1)^\infty\to BU.$$
Let $$Q(x)= \frac{x}{g(x)}=1+a_1x+a_3z(x)\in E^*[[x]].$$ 
By Hirzebruch's theory of multiplicative sequences, in $E^*[[ y_1,y_2,\dots ]]$ the element $\frac x{x'}$ corresponds to $\prod_{k=1}^{\infty}Q(y_k)$, which is symmetric in the $y_k$, so that it defines an element
in $E^*[[ \sigma_j(y_1,y_2,\dots ) \mid j\ge 1]]\cong E^*BU$. Here $\sigma_j(y_1,y_2,\dots )$ is the $j$--th elementary symmetric polynomial of the $y_i$.

We now consider the restriction to $BString$ under the map $BString\to BSO\stackrel{c}\to BU$.
Unstably, we consider a maximal torus $\T'\cong U(1)^{2N+1}$ in $U(2N+1)$ and its Weyl group $W\cong\Sigma_{2N+1}$. 
We have $$E^*B\T'\cong E^*[[y_1,\dots ,y_{2N+1}]],$$ where $y_i$ is the Chern class of the canonical line bundle $L_i$ over the $i$--th factor,
and $$E^*BU(2N+1)\cong E^*[[c_1,c_2,\dots,c_{2N+1}]]\cong E^*B\T'^W,$$ where $W$ acts on $E^*[[y_1,\dots ,y_{2N+1}]]$ by permuting the $y_i$. 
The Chern classes $c_j=\sigma_j(y_1,y_2,\dots)$ are the elementary symmetric polynomials of the $y_i$.

Comparing the unitary and special orthogonal groups and their maximal tori, we obtain the diagram
$$
\xymatrix{
E^*BU(2N+1)\cong E^*[[c_1,c_2,\dots , c_{2N+1}]]
\ar[d]\ar[r]
&
E^*BSO(2N+1)
\ar[d]
\\
E^*B\T' \cong E^*[[y_1,y_2,\dots , y_{2N+1}]]
\ar[r]
&
E^*B\T \cong E^*[[x_1,x_2,\dots ,x_{N}]]
.}
$$

The maximal torus in $SO(2N+1)$ consists of matrices $$A=\begin{pmatrix} R_{\phi_1}&&&\\&\dots&&\\&&R_{\phi_N}&\\&&&1\end{pmatrix},$$
where $R_\phi=\begin{pmatrix}\cos \phi&-\sin\phi\\ \sin\phi&\cos\phi\end{pmatrix}$.
The standard maximal torus in $U(N)$ consists of diagonal matrices. Under conjugation by 
$$\begin{pmatrix} T&&&\\&\dots &&\\&&T&\\&&&1\end{pmatrix},$$ where $T=\frac1{\sqrt 2}\begin{pmatrix}1&1\\-i&i\end{pmatrix},$
the matrix $A$ corresponds to the diagonal matrix $diag(e^{i\phi_1},e^{-i\phi_1},e^{i\phi_2},e^{-i\phi_2},\dots ,1)$,
so that we can assume that the map $B\T\to B\T'$ is induced by $$U(1)^N\to U(1)^{2N+1},\quad (z_1,\dots ,z_N)\mapsto (z_1,z_1^{-1}, \dots ,  z_N,z_N^{-1},1).$$
On the classifying spaces, the induced map $BU(1)^N\to BU(1)^{2N+1}$ pulls back $L_{2k-1}$ to $L_k$ and  $L_{2k}$ to $\overline{L_k}$ for $k\le N$,
while the pullback of $L_{2N+1}$ is trivial.
It follows that in cohomology,
$$
y_{2k-1}\mapsto x_k,\qquad y_{2k}\mapsto \overline{x_k}=g(x_k)\ \ \ \text{for $k\le N$}\ \ \text{ and} \quad y_{2N+1}\mapsto 0.
$$
For the image of the class $\frac x{x'}$, we have
$$
Q(y_{2k-1})\mapsto Q(x_k)=\frac {x_k}{g(x_k)},\quad 
Q(y_{2k})\mapsto Q(g(x_k))=\frac {g(x_k)}{g(g(x_k))}=\frac {g(x_k)}{x_k}
$$
and 
$Q(y_{2N+1})\mapsto 1$.
It follows that $\prod_{k=1}^{\infty}Q(y_k)\mapsto 1$. 

We conclude that the image of $\frac x{x'}$ in $E^*BString$ is trivial, so that 
$\Z/2$ acts trivially on $r\in E^*(BString)$.

Since the action of $\Z/2$ on $E^*B\T$ sends each $x_i\mapsto \overline{x_i}$, we see that
all $x_i\cdot\overline{x_i}$ and therefore also all Pontryagin classes are invariant elements of $E^*(BString)$.

\section{Computation of a coefficient of the cubical structure}\label{Comp}
We prove the part of Proposition \ref{cubical expansion} which is used in the rest of the paper.
The calculation uses the notation introduced in Section \ref{CubicalStr}.
\begin{prop}
Modulo the ideal $(2,a_1,a_2)$, the cubical structure corresponding to $r_U$ has the form $1+a_3x_0x_1x_2+\text{terms of higher total degree}$.
\end{prop}
\begin{proof}
We compute modulo terms of higher order and use the notation $\sum\limits_{cyc}$ for summation over cyclic permutations of the indices $0, 1, 2$
and $\sum\limits_{sym}$ for summation over all permutations:
\begin{align*}
&r_U=\frac tu = \frac{\left\lvert\begin{matrix}x_0&z_0\\ x_1 & z_1\end{matrix}\right\rvert \left\lvert\begin{matrix}x_1&z_1\\ x_2 & z_2\end{matrix}\right\rvert \left\lvert\begin{matrix}x_2&z_2\\ x_0 & z_0\end{matrix}\right\rvert}{\left\lvert\begin{matrix}x_0&1&z_0\\ x_1 & 1 & z_1\\ x_2 & 1 & z_2\end{matrix}\right\rvert z_0z_1z_2} 
\cdot \frac{x_0x_1x_2(x_0+_F x_1+_Fx_2)}{(x_0+_F x_1)(x_1+_F x_2)(x_2+_F x_0)}\\
\\
&=\frac{\prod_{cyc}(x_0\cdot z(x_1)-x_1\cdot z(x_0))\cdot \prod_{cyc}x_0 \cdot (x_0+_F x_1+_F x_2)  }{\left( \sum_{cyc}x_1\cdot z(x_0)-x_0\cdot z(x_1) \right )\cdot \prod_{cyc}z(x_0)\cdot  \prod_{cyc}(x_0+_Fx_1)}\\
&=\frac{\prod_{cyc}(x_1^2+a_3x_1^5-x_0^2-a_3x_0^5)\cdot \left( \sum_{cyc}x_0+\sum_{cyc} a_3x_0^2x_1^2\right)  }
{\left( \sum_{cyc}x_1x_0^3+a_3x_1x_0^6-x_0x_1^3-a_3x_0x_1^6\right )\cdot \prod_{cyc}(1+a_3x_0^3)\cdot  \prod_{cyc}(x_0+x_1+a_3x_0^2x_1^2)},
\end{align*}
where we have divided numerator and denominator by $x_0^3x_1^3x_2^3$ in the last step.
In the resulting fraction, the two terms of lowest order of the numerator have the form 
\begin{align*}
& \prod_{cyc}(x_1^2-x_0^2)\cdot \left( \sum_{cyc}x_0\right)   \\  +&  \prod_{cyc}(x_1^2-x_0^2)\cdot  \left( \sum_{cyc} a_3x_0^2x_1^2\right)+ 
\left(\sum_{cyc} a_3(x_1^5-x_0^5)(x_2^2-x_1^2)(x_0^2-x_2^2)\right) \cdot \left( \sum_{cyc}x_0\right),
\end{align*}
and for the denominator we obtain the two lowest order terms 
\begin{align*}
&\left( \sum_{cyc}x_1x_0^3-x_0x_1^3 \right )\cdot \prod_{cyc}(x_0+x_1) \\
&+\left( \sum_{cyc}x_1x_0^3-x_0x_1^3 \right )\cdot \left(\sum_{cyc}(x_0+x_1)(x_1+x_2)a_3x_0^2x_2^2\right) \\
&+ \left( \sum_{cyc}x_1x_0^3-x_0x_1^3 \right )\cdot \left(\sum_{cyc} a_3x_0^3\right)\prod_{cyc}(x_0+x_1) 
+\left( \sum_{cyc}a_3(x_1x_0^6-x_0x_1^6) \right )\cdot \prod_{cyc}(x_0+x_1).
\end{align*}
Note that the leading terms of numerator and denominator agree, so that the quotient has the form $1+\frac vw$ where
$v$ is the difference of the terms of second lowest order in numerator and denominator 
and 
$w$ is the common term of lowest order.
Now we do the computation modulo 2 and obtain the equality
\begin{align*}
&\frac v{a_3}= \left( \sum_{sym}x_0^4x_1^2 \right ) \cdot \left( \sum_{cyc} x_0^2x_1^2\right )
+ \left(\sum_{sym}x_0^5x_1^4\right)\cdot \left( \sum_{cyc} x_0\right )\\
&+ \left(\sum_{sym}x_0^3x_1\right )\cdot \left(x_0^2x_1^2x_2^2 +  \sum_{cyc} x_0^3x_1^3+ \sum_{sym}x_0^3x_1^2x_2\right )\\
&+ \left(\sum_{sym}x_0^3x_1\right )\cdot \left(\sum_{sym}x_0^5x_1+x_0^4x_1^2+x_0^3x_1^2x_2\right )
+   \left(\sum_{sym} x_0^8x_1^2+x_0^7x_1^3+x_0^7x_1^2x_2+x_0^6x_1^3x_2\right )\\
&= \sum_{sym} x_0^5x_1^4x_2+ x_0^5x_1^3x_2^2+ x_0^6x_1^3x_2 = x_0x_1x_2 \cdot  \left(\sum_{sym} x_0^4x_1^3+ x_0^4x_1^2x_2^1+ x_0^5x_1^2\right )\\
&= \  x_0x_1x_2 \cdot  \left(\sum_{sym} x_0^4x_1^2\right )\cdot \left( \sum_{cyc}x_0\right)=x_0x_1x_2w.
\end{align*}
\end{proof}

\bibliographystyle{amsalpha}

\bibliography{toda}

\end{document}